\documentclass[11pt]{article}
\usepackage[british]{babel}
\usepackage[T1]{fontenc}
\usepackage[utf8]{inputenc}
\usepackage{amsthm}
\usepackage{amsmath}
\usepackage{amssymb}
\usepackage{palatino}

\usepackage{bbm} 

\usepackage{graphicx}
\usepackage[colorlinks=true, citecolor=blue]{hyperref}
\usepackage{dsfont}
\usepackage[all]{xy}
\usepackage[a4paper, margin=3cm]{geometry}
\usepackage{enumitem}
\usepackage{mathtools}
\usepackage{subcaption}
\usepackage{authblk}
\usepackage{comment}

\usepackage{cleveref}

\newcommand{\R}{\mathds{R}}
\newcommand{\N}{\mathds{N}}

\def\d{\,\mathrm{d}}
\def\da{\d a}
\def\ds{\d s}

\def\1{\mathbbm{1}}

\def\:{\colon}

\theoremstyle{plain}

\newtheorem{thm}{Theorem}
\newtheorem{lem}{Lemma}
\newtheorem{prp}{Proposition}

\newtheorem{hyp}{Hypothesis}
\newtheorem{rmk}{Remark}

\numberwithin{equation}{section}

\title{Comparison principles and asymptotic behavior of delayed age-structured neuron models}

\author[1]{María J. Cáceres\thanks{\textit{E-mail address}: \texttt{\href{mailto:caceresg@ugr.es}{caceresg@ugr.es}}}}
\author[1]{José A. Cañizo \thanks{\textit{E-mail address}: \texttt{\href{mailto:canizo@ugr.es}{canizo@ugr.es}}}}
\author[2,1]{Nicolas Torres\thanks{\textit{E-mail address}:\texttt{\href{mailto:torres@ugr.es}{torres@ugr.es}}}}
\affil[1]{Universidad de Granada, Departamento de Matemática Aplicada \& IMAG, Andalusia, Spain.}
\affil[2]{Université Côte d'Azur, Laboratoire Jean Alexandre Dieudonné - CNRS UMR 7351, Nice, France.}

\date{February 2025}

\begin{document}
\maketitle

\begin{abstract}
  In the context of neuroscience the elapsed-time model is an
  age-structured equation that describes the behavior of
  interconnected spiking neurons through the time since the last
  discharge, with many interesting dynamics depending on the type of
  interactions between neurons. We investigate the asymptotic behavior
  of this equation in the case of both discrete and distributed delays
  that account for the time needed to transmit a nerve impulse from
  one neuron to the rest of the ensemble. To prove the convergence to
  the equilibrium, we follow an approach based on comparison
  principles for Volterra equations involving the total activity,
  which provides a simpler and more straightforward alternative
  technique than those in the existing literature on the elapsed-time
  model.
\end{abstract}

\noindent{\makebox[1in]\hrulefill}\newline


\section{Introduction}
\label{sec:intro}

Several mean-field models have been proposed to describe the electrical
activity of a large group of interconnected neurons. They usually take
the form of a partial differential equation with a time variable and
additional variables, often called structure variables, which describe
one or more additional quantities of the system. For example, models
structured by the membrane potential of neurons such as the integrate-and-fire systems are well-known with a vast literature
\cite{brunel2000dynamics,caceres2011analysis,perthame2013voltage,caceres2018analysis,caceres2019global,roux2021towards,caceres2019global,caceres2025asymptotic,caceres2024sequence}.

This article is devoted to the study of an age-structured model for an
interconnected ensemble of neurons described by the elapsed time since
last discharge at the membrane potential, which is known as the
elapsed-time equation (ET). In this model neurons are subjected to
random discharges so that when they reach the firing potential, they
stimulate (or inhibit) other neurons to spike. Depending on the type
of interaction, different possible behaviors of the brain activity are
possible.

This equation was initially proposed in \cite{pakdaman2009dynamics}
and then subsequently developed by many authors with different
extensions by incorporating new elements such as the fragmentation
equation \cite{pakdaman2014adaptation}, spatial dependence with
connectivity kernel in \cite{torres2020dynamics}, a multiple-renewal
equation in \cite{torres2022multiple} and a leaky memory variable in
\cite{fonte2022long}. Moreover, like the case of membrane potential
models such as the Fokker-Planck equation, this model can be obtained
as a mean-field limit of a microscopic model and it establishes a
bridge of the dynamics of a single neuron with a population-based
approach, whose aspects have been investigated in
\cite{pham1998activity,ly2009spike,chevallier2015microscopic,chevallier2015mean,quininao2016microscopic,schwalger2019mind}. Readers
seeking further information may consult \cite{carrillo2025nonlinear}
for a comprehensive review of nonlinear partial differential equations
in neuroscience.

We begin by introducing the model and summarizing its background, followed by a description of the results addressed in this article.

In all models in this paper, $n = n(t,a)$ represents the density at
time $t \geq 0$ of neurons which fired $a \geq 0$ units of time
ago. The time elapsed since the last spike is commonly referred to as the \emph{neuron’s age}.  We always write the models in dimensionless form to simplify the
mathematical treatment, but units can be easily added by standard
procedures. In this work we focus on the \emph{elapsed-time model with distributed delay}, which correspond to the nonlinear system given by
\begin{subequations}
  \label{distri-delay-intro}
  \begin{align}
    \label{distri-delay-transport-intro}
    &\partial_t n +\partial_a n + S(a,X(t))n=0,
    && t,a>0,
    \\
    \label{distri-delay-boundary-intro}
    &n(t,a=0) = r(t) \coloneqq
      \int_0^\infty S(a,X(t))n(t,a)\da,
    && t>0,
    \\
    \label{distri-delay-X-intro}
    &X(t)=\int_{-\infty}^t \alpha(t-s)r(s)\ds,
    && t>0.
  \end{align}
  
  The quantity $r = r(t)$ represents the total number (or density) of
  neurons which fire at time $t$, which means that the membrane
  potential reaches a threshold value and then resets to a baseline
  value. This term $r(t)$ determines the \emph{total activity} of the
  neuron network, represented by the quantity $X = X(t)$ through a
  convolution with a certain nonnegative function
  $\alpha\in L^1(\R^+)$ with $\int_0^\infty\alpha(s)\ds=1$, which is
  known as the kernel of distributed delay. This convolution takes into
  account the delay in transmission after a neuron spikes and the
  value $\alpha(s)$ represents the influence in the total activity at
  time $t$ of a neuron which fired at time $t-s$. In this context, it
  is understood that the history of the rate $r(t)$ for $t < 0$ is
  fixed as an initial condition (as we explain later in
  \eqref{eq:delay-initial-r})

  The nonnegative function $S(a,X)$ is called the \emph{firing coefficient}
  and it represents the susceptibility of neurons to discharge. This
  function accounts for the effect that a total activity $X$ has on
  neurons of age $a$. As we see in the boundary condition
  \eqref{distri-delay-boundary-intro} of $n$ at $a=0$, when a neuron
  discharges at time $t$ its age is reset $0$, so that the firing rate
  $r(t)$ is determined by an integral involving the firing coefficient $S$ and
  the total activity $X(t)$, which depends on the previous states of
  the system for the firing rate and the delay kernel $\alpha$.

  A typical choice for the firing coefficient is
  $S(a, X) = \varphi(X) \1_{a > \sigma}$, which represents a network
  of neurons with an absolute refractory time $\sigma \geq 0$ during which they cannot fire again after a given discharge. Furthermore, the function $S$ may be
  increasing or decreasing in $X$, to allow for excitatory of
  inhibitory interactions respectively and it determines the type of regime of the system. In the case that $S$ does not
  depend on $X$ the model becomes linear, and its study is
  considerably simpler. We notice that $r(t)$ can be calculated by
  knowing $n(s,a)$ for times $s<t$, so equation
  \eqref{distri-delay-X-intro} is a type of delayed boundary
  condition.

The above equation should be complemented by a suitable
  initial condition,
\begin{align}
  \label{eq:delay-initial-n}
  n(t=0, a) = n^0(a), &\qquad a > 0,
  \\
  \label{eq:delay-initial-r}
  r(t) = r^0(t), &\qquad t < 0,
\end{align}
\end{subequations}
where $n^0 \in L^1(\R^+)$ is a given nonnegative function, and $r^0$
is defined on $(-\infty,0)$. Since \eqref{distri-delay-intro} is a
delay equation, it would be natural to specify $n(t,a)$ for
$t \in (-\infty,0]$ as an initial condition, but only the firing rate
$r(t) = n(t,0)$ is actually used, so we emphasize that it is enough to
set $r(t)$ for negative times $t$. Thus we allow for "infinite
delay" in the equation. The statement of the model in
\cite{pakdaman2009dynamics} is equivalent to assuming $r(t) = 0$ for
all $t < 0$. If for a certain $d > 0$ one assumes that $\alpha(t) = 0$
for all $t > d$, then it is clearly enough to give $r(t)$ for
$t \in [-d,0]$ as initial data (since the values of $r(t)$ for
$t < -d$ do not play any role).

Moreover, we formally have the following mass-conservation property
\begin{equation}
\label{massconservation}
    \int_0^\infty n(t,a)\,da = \int_0^\infty n^0(a)\,da,\qquad\forall t\ge 0,
\end{equation}
and without loss of generality, we will normalize it to $1$ so that $n(t,\cdot)$ can be interpreted as the probability distribution at time $t$ of the time since the last spike.

There are two important situations which are limiting cases of this
one. First, if we take the limit as $\alpha \to \delta_{d}$ (a Dirac
delta function at $t=d$) for some $d > 0$ we formally obtain the model
with \emph{single discrete delay}:
\begin{subequations}
  \label{single-delay-intro}
  \begin{align}
    \label{single-delay-transport-intro}
    &\partial_t n +\partial_a n + S(a, r(t-d))n=0,
    && t,a>0,
    \\
    \label{single-delay-boundary-intro}
    &n(t,a=0) = r(t) \coloneqq
      \int_0^\infty S(a,r(t-d)) n(t,a)\da.
    && t>0.
  \end{align}

This system is known as the case with discrete delay, where the total activity is just the firing rate at time
$t-d$. Now the natural initial condition involves setting
\begin{align}
  \label{eq:single-delay-initial-n}
  n(t=0, a) = n^0(a), &\qquad a > 0,
  \\
  \label{eq:single-delay-initial-r}
  r(t) = r^0(t), &\qquad -d \leq t < 0.
\end{align}
\end{subequations}
In turn, if we consider the limit $d = 0$ then this system becomes
\begin{subequations}
  \label{no-delay-intro}
  \begin{align}
    \label{no-delay-transport}
    &\partial_t n +\partial_a n + S(a, r(t))n=0,
    && t,a>0,
    \\
    \label{no-delay-boundary}
    &n(t,a=0) = r(t) =
      \int_0^\infty S(a,r(t)) n(t,a)\da.
    && t>0.
  \end{align}
\end{subequations}
This system is known as the case with instantaneous transmission. Now the definition of $r(t)$ is an independent equation, which has to
be solved together with the whole system and the only initial
condition to set is $n(0,a)$ for $a > 0$. If $n(t,a)$ is known for a
certain $t$, then finding an $r(t)$ which satisfies
$r(t) = \int_0^\infty S(a,r(t)) n(t,a)\da$ may be an ill-posed
problem; see \cite{pakdaman2009dynamics} or
\cite{canizo2019asymptotic} for a simple example, and more recently
\cite{sepulveda2023well} for an analysis of the conditions which may
stop this system from being well-posed.

Concerning the steady states, the equilibria of
\eqref{distri-delay-intro} are given by the equation:

\begin{equation}
  \label{stationary-equation}
    \begin{cases}
           \partial_a n^* + S(a,X^*)n^*=0 & a>0,\vspace{0.1cm}\\
            r^*\coloneqq n^*(a=0)=\int_0^\infty S(a,X^*)n^*(a)\da,&\vspace{0.1cm}\\
            X^*=r^*\int_0^\infty\alpha(s)\ds.
    \end{cases}
\end{equation}
Thanks to the normalization $\int_0^\infty\alpha(s)\,ds=1$ we have
$X^*=r^*$. From the first equation of the system we get that
$n^*(a):=r^*e^{-\int_0^a S(a',X^*) \ da'}$ and the following equation
holds for $r^*$:
\begin{equation}
\label{Int-r}
    r^*I(r^*)=1,\ \textrm{with}\:I(r)\coloneqq \int_0^\infty e^{-\int_0^a S(s,r)\ds}\da,
\end{equation}
as a consequence of the mass-conservation property. 

For simplicity we call the steady state as the pair $(n^*,r^*)$, since $X^*=r^*$. Moreover, for the case of instantaneous transmission \eqref{no-delay-intro} and the case with discrete delay \eqref{distri-delay-intro} the definition of an equilibrium is analogous and in all cases we have the same steady states for a given firing coefficient $S$. We also remark that when the system is inhibitory it has a unique steady state, while in the excitatory case multiple steady states may arise \cite{pakdaman2009dynamics}.

Finally, if we fix $X=\bar{r}\ge0$ as parameter in the coefficient $S$ we obtain the following linear equation, which is fundamental to understand the
non-linear problems \eqref{distri-delay-intro} and
\eqref{single-delay-intro}.
\begin{equation}
  \label{eqlinear}
  \left\{
    \quad
  \begin{aligned}
    &\partial_t n +\partial_a n + S(a,\bar{r})n=0 & t,a>0,\\
    &n(t,a=0) = r(t)\coloneqq\int_0^\infty S(a,\bar{r})n(t,a)\da& t>0,\\
    &n(t=0,a)=n^0(a)& a>0.
  \end{aligned}
  \right.
\end{equation}  
Observe that this linear equation does not have any explicit delay and in this case it can be cast in the form of an
abstract ODE in the space $\mathcal{M}(\R^+)$ of finite signed Borel
measures, given by
\begin{equation}
  \label{operator-L}
  \partial_t n
  =
  L_{\bar{r}}[n]\coloneqq-\partial_a n - S(a,\bar{r})n+\delta_0 \int_0^\infty S(a,\bar{r})n(a)\da.
\end{equation}
For the sake of simplicity of the notation in the computations, we
treat the elements in $\mathcal{M}(\R^+)$ as if they were integrable
functions with corresponding generalization. The solution of this
linear problem determines a positive and mass-preserving semigroup in
$\mathcal{M}(\R^+)$, which will be denoted as $e^{tL_{\bar{r}}}$ in the
sequel. In other words, $e^{tL_{\bar{r}}}$ is a Markov semigroup. The
asymptotic behavior of $e^{tL_{\bar{r}}}$ is well-known, as we state in the following result. 

\begin{prp}[Linear spectral gap]
\label{doeblin-conv}
Assume that $S$ satisfies Hypothesis \ref{hyp:S}, and let $\bar{r}
\geq 0$ be given. Then the pair 
$\left(\bar{n}^*:=\bar{r}^*e^{-\int_0^a S(s,\bar{r}) \ ds},
\bar{r}^*:=\left(\int_0^\infty e^{-\int_0^a S(s,\bar{r}) \ ds}\ da\right)^{-1}\right)$
is the unique positive stationary solution to  Equation
\eqref{eqlinear} such
that $n^*\in L^1(0,\infty)$ with $\int_0^\infty n^*\da=1$.
And there
exist constants $C_0, \lambda>0$ such that for all initial data
$n^0\in \mathcal{M}(\R^+)$ it holds that, for all $t\ge0$,
   \begin{equation}
     \label{exp-conv-linear}
 \begin{aligned}           
&   \|e^{tL_{\bar{r}}} n^0-\langle n^0\rangle \bar{n}^*\|_{TV}
   \le
   C_0 e^{-\lambda t}\|n^0-\langle n^0\rangle  \bar{n}^*\|_{TV}
        \\
        & |r(t)-\langle n^0\rangle \bar{r}^*|
        \le
        C_0 e^{-\lambda t}\|n^0-\langle n^0\rangle  \bar{n}^*\|_{TV}
      \end{aligned}
   \end{equation}
   with $\langle n^0\rangle\coloneqq \int_0^\infty n^0 \da$.
 \end{prp}
We remark that the constant $\lambda$ gives the natural speed of convergence to equilibrium of \eqref{eqlinear}. This result can be proved through different techniques such
as the entropy method \cite{perthame2006transport}, Doeblin's theory
\cite{canizo2019asymptotic} and Kato's inequality
\cite{mischler2018weak}.
 
\medskip Concerning the nonlinear case, global well-posedness of weak solutions has been studied in the case with instantaneous transmission and also distributed delay \cite{pakdaman2009dynamics,chevallier2015mean,mischler2018weak,canizo2019asymptotic} and more recently in \cite{sepulveda2023well} with a numerical scheme inspired in fixed-point problems.

Regarding long-time behavior, global results are comparatively rare: no general results on convergence to equilibrium are available, and no useful entropy or Lyapunov functional is known for the nonlinear model. Some partial results in this direction include \cite{torres2021elapsed}, where the existence of periodic solutions with jump discontinuities was established in the case of strong non-linearities.

However, a quite complete analysis can be carried out in perturbative
situations, when the system is close to a linear system. 
In this regard, the following properties are expected  to hold:
\begin{enumerate}
\item There exists a unique probability equilibrium $n^*$, with its associated firing rate $r^*$.
\item All solutions with an initial  probability distribution converge to this equilibrium  as $t \to +\infty$ at an exponential rate.
\end{enumerate}
It is reasonable to study these properties when the nonlinearity is
weak, meaning that the following holds for a small enough $\ell$:
\begin{equation}
  \label{eq:2}
  |S(a, r) - S(a,r')| \leq \ell |r-r'|
  \qquad \text{for all $a, r, r' > 0$.}
\end{equation}
In previous papers a very similar condition to this is always used. In
the literature on neuron dynamics this condition corresponds to either
``weak connectivity'' or ``strong connectivity'' regimes. To
understand the meaning of weak and strong regimes it is important to
notice that the firing coefficient is usually written as
$\widetilde{S}(a,JX)$ in other references, where $J\ge0$ is the
network connectivity parameter.  We have avoided this notation to
simplify the presentation of the model, so we do not have a parameter
$J$; we just take $S(a, X) \equiv \widetilde{S}(a,JX)$. When using
this notation, ``weak connectivity'' corresponds to small $J$, and
``strong connectivity'' corresponds to large $J$. With appropriate
(additional) assumptions on $\widetilde{S}$, one may show \eqref{eq:2} (or a very
similar condition to \eqref{eq:2}) either when $J$ is small enough, or
when $J$ is large enough.

Results on properties 1 and 2 were first given in
\cite{pakdaman2009dynamics,pakdaman2013relaxation,pakdaman2014adaptation}
by using variations of the generalized relative entropy method
\cite{michel2005general,perthame2006transport} in the case of
instantaneous transmission \eqref{no-delay-intro}, while a semigroup
approach based on Doeblin's theory
\cite{bansaye2020ergodic,gabriel2018measure} was given in
\cite{canizo2019asymptotic}, applicable to both equation
\eqref{no-delay-intro} and modified models with fatigue proposed in
\cite{pakdaman2014adaptation}. The same ideas were also used to study
a model structured by additional past discharge times in
\cite{torres2022multiple} and with a memory term in
\cite{fonte2022long}.

Besides the case with instantaneous transmission, exponential
convergence with distributed delay has been previously studied by
Mischler et al. \cite{mischler2018,mischler2018weak} for weak
nonlinearities under regularity assumption such as when the firing
coefficient $S\in\mathop{\textrm{Lip}}_X\,L^1_a$. This result was
proved through a spectral analysis based on the analysis in
\cite{mischler2016spectral} for the growth-fragmentation equation.

The goal of our article is to fill some gaps on the convergence to the
equilibrium for the elapsed-time model for both distributed and
discrete delays under the regime of weak nonlinearities with an
alternative method to the spectral analysis previously cited and under
simple assumptions for $S$. Apart from introducing a relatively simple
method, we are also able to include some new cases such as the case of
algebraic decay of $\alpha$. In addition, the case with discrete delay, despite
being formally a particular case of the distributed delay, should be
analyzed separately, since it is a singular case and no explicit proof
was given before for this case.

Our approach relies on a comparison principle for integral equations
involving the distance to equilibrium of the total activity
$|X(t)-X^*|$ in the case of Equation \eqref{distri-delay-intro} and
$|r(t)-r^*|$ in the case of Equation \eqref{single-delay-intro}. The
strategy consists in finding a suitable upper solution of a
Volterra-type equation that vanishes when $t\to\infty$ and that allows
to bound the quantities $|X(t)-X^*|$ and $|r(t)-r^*|$. Comparisons
techniques for other age-structured models have been recently studied
in \cite{deng2020analysis} with logistic growth and spatial diffusion.

The advantage of this argument is that we obtain a simpler proof of convergence to equilibrium, whose rate also depends explicitly on the bounds of the kernel $\alpha$ and the delay $d$ in their respective cases. Moreover, we also point out that suitable modifications of the argument based on a perturbation of the linear case stated in Proposition \ref{doeblin-conv} can also deal
with the delayed equations \eqref{distri-delay-intro} and
\eqref{single-delay-intro}, and get the desired property 2 above.

\subsection{Main results of this article}

We now present the results of this paper, highlighting the
crucial role played by the size $\lambda$ of the \emph{spectral gap}
of the linear equation \eqref{eqlinear}, with $\bar{r}=r^*$, given in 
\Cref{doeblin-conv}. We will always assume the
following:
\begin{hyp}[Conditions on $S$]
  \label{hyp:S}
  We assume $S\colon(0,+\infty) \times [0,+\infty) \to [0,+\infty)$ is a
  bounded measurable function, 
  and 
  Lipschitz with respect to its second variable with Lipschitz constant $l$:
  \begin{equation*}
    |S(a, r) - S(a,r')| \leq \ell |r-r'|
    \qquad \text{for all $a, r, r'>0$.}
  \end{equation*}
  We also assume that there exist constants $s_0,\sigma>0$ such that
  \begin{equation}
    \label{boundS}
    S(a,r) \geq s_0\1_{\{a>\sigma\}}
    \qquad \text{for all $a,r\ge 0$.}
  \end{equation}
\end{hyp}

\begin{hyp}[Initial conditions]
  \label{hyp:initial_condition}
  We assume that $n^0$ is a nonnegative probability measure on
  $(0,+\infty)$. For the distributed delay equation
  \eqref{distri-delay-intro} we assume that $r^0 \,: (-\infty,0] \to
  [0,+\infty)$ is a bounded function; for the single discrete delay
  equation \eqref{single-delay-intro} we assume that $d > 0$ and $r^0 \,: [-d,0] \to
  [0,+\infty)$ is a bounded function.
\end{hyp}

It is also known, in general, \cite{pakdaman2009dynamics,
  pakdaman2013relaxation, canizo2019asymptotic} that in either weak or strong
connectivity regime the nonlinear problems \eqref{distri-delay-intro}
and \eqref{single-delay-intro} have a unique probability equilibrium:
there exists $\ell_* > 0$ such that if $S$ satisfies \Cref{hyp:S} with
$0 \leq \ell \leq \ell_*$ then equations \eqref{distri-delay-intro}
and \eqref{single-delay-intro} have a unique equilibrium $(n^*, r^*)$
such that $n^*$ is a probability measure. Since our results below are
stated for small $\ell$ one may always assume that
$\ell \leq \ell_*$, so the fact that there is a unique equilibrium in
that case is known. The results presented in this article are still valid when $S$ satisfies similar Lipschitz estimates involving the integral with respect to $a$, as it was done for example in \cite{mischler2018,mischler2018weak}. Furthermore, see Remarks \ref{rmkDiscret} and \ref{rmkDistri} 
for more details on how to apply our main results in the context of weak and strong regimes.

The following are the main results of this article. Regarding the single discrete delay model we have:

\begin{thm}[Single discrete delay]
  \label{delay-conv-thm}
  Assume Hypothesis \ref{hyp:S}, with $\ell$ small enough such that 
  there exists a unique steady state $(n^*,r^*)$ of equation
    \eqref{single-delay-intro}, and let $\lambda > 0$ be the spectral
    gap of the linear equation \eqref{eqlinear}, with $\bar{r}=r^*$.
    Then there exists $\ell_0 > 0$ (depending only on
  $\lambda$) such that for all $d>0$ there exist constants $0<\mu<\lambda$ (depending only on $d$ and $\lambda)$, $C>0$ so that when $\ell\le\ell_0$ any initial condition $(n^0, r^0)$ satisfying Hypothesis \ref{hyp:initial_condition}, the solution
    $(n,r)$ of equation \eqref{single-delay-intro} satisfies
    \begin{equation}
      \label{conv-delay-intro}
      \begin{aligned}
        &\|n(t)-n^*\|_{TV}\le C K_0 e^{-\mu t},
        \\
        &|r(t)-r^*|\le C K_0 e^{-\mu t}
      \end{aligned}
    \end{equation}
    for all $t > 0$, where $K_0$ measures the initial distance to
    equilibrium in the following sense:
    \begin{equation*}
      K_0 := \|r^0 - r^*\|_{\infty} + \|n^0 - n^*\|_{TV}.
    \end{equation*}
 
\end{thm}
We notice that in this case $\|r^0 - r^*\|_\infty$ denotes the
$L^\infty$ norm in the interval $[-d,0]$.

\medskip The previous theorem informally states that in the
weak-connectivity regime, the nonlinear model
\eqref{single-delay-intro} converges to equilibrium at essentially the
same rate as the linear system. We can also obtain similar results for
the distributed delay model \eqref{distri-delay-intro}, with the
important difference that solutions will now converge to equilibrium
at (roughly) the slowest of the following rates:
\begin{enumerate}
\item The rate $e^{-\lambda t}$ of decay to equilibrium of the linear model.
  
\item The decay rate to $0$ of the function $\alpha$.
\end{enumerate}
The following two results make this idea precise:

\begin{thm}[Exponentially distributed delay]
  \label{distri-conv-thm}
  Assume Hypothesis \ref{hyp:S}, with $\ell$ small enough such that 
  there exists a unique steady state $(n^*,r^*)$ of equation
    \eqref{distri-delay-intro}, and let $\lambda > 0$ be the spectral
    gap of the linear equation \eqref{eqlinear}, with $\bar{r}=r^*$. Assume that there exist
  constants $C_\alpha,\beta>0$ such that
  $$
  \alpha(t)\le C_\alpha e^{-\beta t}
  \qquad \text{for all $t > 0$.}
  $$
  Then, for any $0 < \mu < \min\{\lambda, \beta\}$ there exists
  $\ell_0 > 0$ depending only on $\|S\|_\infty$ and $\mu$ such that if
  $\ell \le \ell_0$,  there exists a constant $C > 0$ (depending only on $S$,
    $C_\alpha$ and $\beta$) such that for any initial condition
    $(n^0, r^0)$ satisfying Hypothesis \ref{hyp:initial_condition} the
    solution $(n,r)$ of equation \eqref{distri-delay-intro} satisfies
      \begin{align}
        \label{conv-delay-distri-exp}
        &\|n(t)-n^*\|_{TV}\le C K_0 e^{-\mu t},
        \\
        \label{conv-delay-distri-exp-r}
        &|r(t)-r^*|\le C K_0 e^{-\mu t},
        \\
        \label{conv-delay-distri-exp-X}
        &|X(t)-X^*|\le C K_0 e^{-\mu t}
    \end{align}
    for all $t > 0$, where $K_0$ measures the initial distance to
    equilibrium in the following sense:
    \begin{equation*}
      K_0 := \|r^0 - r^*\|_{\infty} + \|n^0 - n^*\|_{TV}.
    \end{equation*}
\end{thm}
In this case $\|r^0 - r^*\|_{\infty}$ denotes the $L^\infty$ norm on
$(-\infty, 0)$. We also point out that $X^* := r^* \int_0^\infty
\alpha(s) \d s$ is the total activity at equilibrium.

Regarding algebraic tails we have a similar result, this time with an
algebraic speed of convergence:

\begin{thm}[Distributed delay, algebraic tail]
  \label{distri-conv-thm-noexp}
  Assume Hypothesis \ref{hyp:S},  with $\ell$ small enough such that 
  there exists a unique steady state $(n^*,r^*)$ of equation
    \eqref{distri-delay-intro}, and let $\lambda > 0$ be the spectral
    gap of the linear equation \eqref{eqlinear}, with $\bar{r}=r^*$.
    Assume that there exist constants $C_\alpha>0$, $\beta>1$ such that
  $$\alpha(t)\le \frac{C_\alpha}{1+t^\beta}.$$
  Then there exists $\ell_0 > 0$ depending only on $S$ such that if
  $\ell \le \ell_0$,  there exists a constant $C > 0$ (depending only on $S$,
    $C_\alpha$ and $\beta$) such that for any initial condition
    $(n^0, r^0)$ satisfying Hypothesis \ref{hyp:initial_condition} the
    solution $(n,r)$ of equation \eqref{distri-delay-intro} satisfies
    \begin{align}
      \label{conv-delay-distri-exp2}
      &\|n(t)-n^*\|_{TV}\le \frac{C K_0}{1+t^{\beta-1}},
      \\
      \label{conv-delay-distri-exp-r2}
      &|r(t)-r^*|\le \frac{C K_0}{1+t^{\beta-1}},
      \\
      \label{conv-delay-distri-exp-X2}
      &|X(t)-X^*|\le \frac{C K_0}{1+t^{\beta-1}}
    \end{align}
    for all $t > 0$, where $K_0$ measures the initial distance to
    equilibrium in the following sense:
    \begin{equation*}
      K_0 := \|r^0 - r^*\|_{\infty} + \|n^0 - n^*\|_{TV}.
    \end{equation*}
\end{thm}

This result allows to extend the convergence result in \cite{mischler2018,mischler2018weak} where the kernel $\alpha$ must have a Laplace transform $\widehat{\alpha}(z)$ defined for $\Re(z)>-c$ for some $c>0$, i.e. $\alpha$ decays exponentially. Thus, even if $\alpha$ decays like an inverse of a polynomial, it is still possible to have convergence to the equilibrium with explicit rates that depend on the bounds of $\alpha$.

The proof of the above results is based on a perturbation argument,
writing the nonlinear equations as the linear one plus a perturbation
term which can be shown to be small, and then using Duhamel's formula
to compare with the solution of the linear equation. There are two
important ideas to consider in order to carry out this plan: first,
it is natural to consider the spectral gap in total variation
norm, as the perturbation term is small in this norm (but is not even
finite in stronger norms such as $L^p$); this was used in
\cite{canizo2019asymptotic} in order to study the case without
delay. Second, the inequalities obtained after using Duhamel's
formula are modified versions of Volterra integral equations for which
there is no general theory readily available. We give comparison
theorems for them, from which one can then obtain the main results.

The rest of the paper is devoted to proving the convergence theorems and offering remarks and perspectives that emerge from them. It is organized as follows: Section~\ref{sec:single-delay} contains
the proof of Theorem \ref{delay-conv-thm}, while Section \ref{sec:distri-delay}
contains the proofs of Theorems \ref{distri-conv-thm} and
\ref{distri-conv-thm-noexp}.

\section{Model with a single discrete delay: Proof of Theorem \ref{delay-conv-thm}}
\label{sec:single-delay}

This section is devoted to the elapsed time equation with a single
discrete delay given in \eqref{single-delay-intro}:
\begin{subequations}
   \label{single-delay}
  \begin{align}
    \label{single-delay-transport}
    &\partial_t n +\partial_a n + S(a, r(t-d))n=0,
    && t,a>0,
    \\
    \label{single-delay-boundary}
    &n(t,a=0) = r(t) \coloneqq
      \int_0^\infty S(a,r(t-d)) n(t,a)\da,
    && t>0.
  \end{align}
  \end{subequations}

We remind that the steady states $(n^*,r^*)$ in this case are given by:
\begin{equation*}
    \begin{cases}
           \partial_a n^* + S(a,r^*)n^*=0 & a>0,\vspace{0.1cm}\\
            r^*\coloneqq n^*(a=0)=\int_0^\infty S(a,r^*)n^*(a)\,da,&\\
    \end{cases}
\end{equation*}
where $n^*(a)=r^*e^{-\int_0^a S(a',r^*) \da'}$ and $r^*>0$ satisfies Equation \eqref{Int-r}.

The aim of this section is to prove Theorem \ref{delay-conv-thm}.
To achieve this, we make use of the following comparison lemma.

\begin{lem}[Comparison lemma with discrete delay]
  \label{comparison-delay-simple}
  Consider the constants $d>0,c_1\ge0,c_2\ge0$ and the functions
  $f\in L^\infty(0,\infty),\,u^0\in L^\infty(-d,0)$. Let
  $\underline{u}\in L^\infty(-d,\infty)$ such that
  \begin{equation}
    \begin{cases}
      \underline{u}(t)\le c_1 \underline{u}(t-d)
      + c_2 \int_0^t e^{-\lambda (t-s)}\underline{u}(s-d)\ds
      +f(t)& \forall t>0,\\
      \underline{u}(t)\le u^0(t) & \forall t\in(-d,0),
    \end{cases}
  \end{equation}
  and $\overline{u}\in L^\infty(-d,\infty)$ such that 
  \begin{equation}
    \begin{cases}
      \overline{u}(t)\ge c_1 \overline{u}(t-d)
      + c_2 \int_0^t e^{-\lambda (t-s)}\bar{u}(s-d) \ds+f(t)& \forall t>0,\\
      \overline{u}(t)\ge u^0(t) &\forall t\in(-d,0),
    \end{cases}
  \end{equation}
  Then $\underline{u}(t)\le\overline{u}(t)$ for all $t>-d$.
\end{lem}

In other words $\underline{u}$ and $\overline{u}$ are respectively lower and upper solutions of the delayed Vol\-te\-rra-type equation given by
\begin{equation}
\begin{cases}
 u(t)= c_1 u(t-d)
      + c_2 \int_0^t e^{-\lambda (t-s)}u(s-d) \ds+f(t)& \forall t>0,\\
      u(t)= u^0(t) &\forall t\in(-d,0),    
\end{cases}
\end{equation}
and the comparison principle holds.

\begin{proof}
    Observe that $h(t)\coloneqq\overline{u}(t)-\underline{u}(t)$ satisfies the following inequalities
    \begin{equation*}
        \begin{cases}
           h(t)\ge c_1 h(t-d) + c_2 \int_0^t e^{-\lambda (t-s)}h(s-d)ds& \forall t>0,\\
           h(t)\ge 0 & \forall t\in(-d,0).
        \end{cases}
    \end{equation*}
    From the first inequality we conclude that $h(t)\ge0$ for all $t\in(0,d)$ and by iterating over the intervals $(kd,(k+1)d)$ with $k\in\N$, we conclude that $h(t)\ge0$ for all $t>-d$.
\end{proof}

Now we can proceed with the proof of Theorem \ref{delay-conv-thm}.
\begin{proof}[Proof of Theorem \ref{delay-conv-thm}]
  We write the solution of Equation \eqref{single-delay} as
    $$\partial_t n = L_{r^*}[n]+h$$
  where the linear operator $L_{r^*}$ was defined in \eqref{operator-L},
  with $\bar{r}=r^*$, and $h$ is given by
    $$h(t,a)=(S(a,r^*)-S(a,r(t-d)))n(t,a)+\delta_0(a)\int_0^\infty
    (S(a',r(t-d))-S(a',r^*))n(t,a')\d a',$$ and by applying Duhamel's
    formula and \Cref{doeblin-conv}, there exists
    $C_0,\lambda>0$ such that the following inequality holds:
    \begin{equation}
    \label{est-delay-n-n*}
    \|n(t)-n^*\|_{TV}\le C_0e^{-\lambda t}\|n^0-n^*\|_{TV}+C_0\int_0^t
    e^{-\lambda(t-s)}\|h(s)\|_{TV} \ds.
    \end{equation}
    By using the mass-conservation property \eqref{massconservation}
    of the system, for $h$ we obtain
    \begin{equation}
      \label{est-h-delay}
      \|h(t,\cdot)\|_{TV}\le 2 \ell |r(t-d)-r^*|\qquad\forall t>0,
    \end{equation}
    where $\ell$ is the Lipschitz constant of $S$ with respect to $r$
    (see Hypothesis \ref{hyp:S}). Also, from the definition of $r(t)$ (see
    \eqref{single-delay-boundary})
    we obtain
    \begin{multline*}
      |r(t)-r^*|
      =
      \left| \int_0^\infty S(a, r(t-d)) n(t,a) \d a
        - \int_0^\infty S(a, r^*) n^*(a) \d a
      \right|
      \\
      \leq
      \int_0^\infty |S(a, r(t-d)) - S(a,r^*)| n(t,a) \d a
      +
      \int_0^\infty S(a, r^*) |n(t,a) - n^*(a)| \d a
      \\
      \leq
      \ell |r(t-d) - r^*|
      +
      \|S\|_\infty \| n(t,a) - n^*(a) \|_{TV}.
    \end{multline*}
    Now using \eqref{est-delay-n-n*} and \eqref{est-h-delay} in the
    previous equation we get
    \begin{equation}
      \begin{split}
        |r(t)-r^*|
        &\le
        \ell |r(t-d)-r^*|
        + C_0\|S\|_\infty \|n^0-n^*\|_{TV}e^{-\lambda t}
        \\
        & \qquad
        +2C_0\|S\|_\infty
        \ell \int_0^t
        e^{-\lambda(t-s)}|r(s-d)-r^*|\ds
.
        \end{split}
    \end{equation}
    We define the constants
    $\,C_1\coloneqq
    2C_0\|S\|_\infty$ and $\,C_2\coloneqq C_0\|S\|_\infty \|n^0-n^*\|_{TV}$
    so that for $u(t)\coloneqq |r(t)-r^*|$ we get the
    inequality
    $$u(t)\le \ell u(t-d)
    +C_1 \ell \int_0^t e^{-\lambda(t-s)}u(s-d) \ds +C_2 e^{-\lambda
      t}\quad\forall t\ge 0.$$ The main idea is to apply now the
    comparison lemma. We look for a constant $A,\mu>0$ such that we
    get $u(t)\le Ae^{-\mu t}$ for all $t>-d$. This means that the
    function $v(t)\coloneqq Ae^{-\mu t}$ must satisfy the following
    inequalities
    $$\begin{cases}
      v(t)\ge \ell v(t-d)
      +C_1 \ell \int_0^t e^{-\lambda(t-s)}v(s-d) \ds+C_2 e^{-\lambda t}&
      \forall t>0\\
      v(t)\ge |r^0-r^*| &\forall t\in(-d,0),   
    \end{cases}
   $$
   or equivalently in terms of $A$ and $\mu$
    \begin{equation}
    \label{supersol-A-lambda}
        \begin{matrix*}[l]
           \displaystyle A\left(1 - \ell e^{\mu d}-C_1 \ell e^{\mu d}\frac{1-e^{-(\lambda-\mu)t}}{\lambda-\mu}\right)\ge C_2 e^{-(\lambda-\mu)t} & \forall t>0\\
        A\ge e^{\mu t}|r^0(t)-r^*| &\forall t\in(-d,0).   
        \end{matrix*}
    \end{equation}
    
    Observe that (using $e^{-(\lambda-\mu) t} \geq 0$ on the left and
    $e^{-(\lambda-\mu) t} \leq 1$ on the right) a sufficient condition
    to verify \eqref{supersol-A-lambda} is given by the inequalities
    \begin{align*}
      &A \left(
        1- \ell \left(
        e^{\mu d} + C_1e^{\mu d}\frac{1}{\lambda-\mu}
        \right)
        \right)
        \ge C_2
      \\
      &A\ge \sup_{t\in[-d,0]}|r^0(t)-r^*|.  
    \end{align*}
    Therefore, for $\ell > 0$ satisfying 
    \begin{equation*}
    \ell \left(e^{\mu d}+C_1e^{\mu d}\frac{1}{\lambda-\mu}\right)<1,        \:
    \qquad
    \text{or equivalently}
    \qquad
    \ell < \frac{e^{-\mu d}(\lambda-\mu)}{\lambda -\mu+C_1},
    \end{equation*}
    and $A$ verifying
    \begin{equation*}
      A>\max\left\{||r^0-r^*||_\infty,
        \
        \frac{C_2(\lambda-\mu)}
        {\lambda-\mu-\ell_0 e^{\mu d} (\lambda - \mu + C_1 )}
      \right\}
      \qquad
      \text{with $\mu<\lambda$, and $\ell\le\ell_0$,}
\end{equation*}
we get that \eqref{supersol-A-lambda} holds and hence $v(t)$ satisfies
the desired inequalities. By Lemma \ref{comparison-delay-simple} we
conclude that
\begin{equation}
  \label{ut-exp-bounded}
  u(t)=|r(t)-r^*|\le A e^{-\mu t}.
\end{equation}
Without loss of generality we can assume
$\|n^0-n^*\|_{TV}+\|r^0-r^*\|_\infty>0$, so we can choose $A$ of the
form
$A = \widetilde{C}(S,d,\mu) \left(\|n^0-n^*\|_{TV} +
  \|r^0-r^*\|_\infty\right)$ (since $C_2$ is the only constant we
defined which depends on the initial distance $\|n^0-n^*\|_{TV}$).

We now assert that we can find a bound on $\ell$, $\ell_0$, independent of $d$
such that \eqref{ut-exp-bounded} holds for some choice of
$\mu>0$. Indeed, when $d\le1$ we can choose $\mu=\frac{\lambda}{2}$
and set $\ell_0\coloneqq \frac{\lambda e^{-\lambda}}{\lambda+2C_1}$ such
that for
$$\ell \le \ell_0 < \frac{\lambda e^{-\frac{\lambda}{2}}}{\lambda+2C_1}$$
the estimate \eqref{ut-exp-bounded} is verified. Similarly for $d>1$,
if we take $\mu=\frac{\lambda}{d+1}$ such that for
$$\ell \le \ell_0 < \frac{d\lambda e^{-\frac{d}{d+1}\lambda}}{d\lambda+(d+1)C_1},$$
the same conclusion holds. Finally, from estimates \eqref{est-delay-n-n*} and \eqref{est-h-delay} the exponential convergence of $\|n(t)-n^*\|_{TV}$ in
\eqref{conv-delay-intro} readily follows. 
    \end{proof}
In light of the proof, we draw attention to the following remarks.
\begin{rmk}
  We have proved the existence of a sufficiently small connectivity
  parameter $\ell_0$ such that for any transmission delay $d$, we have
  exponential convergence of the system towards its unique steady
  state.  Nevertheless, the rate of convergence is influenced by $d$
  and, as expected, it decreases as $d$ increases.

  An interesting extension would be to jointly study the dependence on the delay $d$ and the spectral gap given $\lambda$ in Proposition \ref{doeblin-conv}. For a given delay $d$, one would expect that a larger value of $\lambda$ will allow a larger value of the Lipschitz constant $\ell_0$ where the exponential convergence holds. The choice of $\ell_0=\frac{\lambda e^{-\lambda}}{\lambda+2C_1}$ obtained in the proof of the theorem is decreasing in terms of $\lambda\gg1$, suggesting that this bound might be improved.
  \end{rmk}

\begin{rmk}[Relaxed hypotheses]
\label{rmkDiscret}
Our proofs can be adapted to work when the following relaxed condition
on $S$ holds, instead of Hypothesis \ref{hyp:S}:
\begin{hyp}[Conditions on $S$]  
  \label{hyp: S-weak-strong}
  We assume $S\colon (0,+\infty) \times [0,+\infty) \to [0,+\infty)$ is a
  bounded measurable function, and let $(n^*, r^*)$ be an equilibrium
  of the linear equation \eqref{eqlinear}. We assume that $S$ is
  Lipschitz with respect to $r$ with constant $\ell$ when $|r-r^*|$ is
  small enough, that is: there exists $\delta > 0$ such that
  \begin{equation*}
    |S(a, r) - S(a,r')| \leq \ell |r-r'|
    \qquad \text{for all $a>0$ and all $r,r' \in [r^*-\delta, r^* + \delta]$.}
  \end{equation*}
  \end{hyp}
  This assumption requires a Lipschitz constant only close to the
  equilibrium $r^*$. With this condition, following the proof of
  Theorem \ref{delay-conv-thm}, we may obtain convergence to the
  equilibrium provided that the initial condition is close enough to
  the equilibrium itself. In this case one does not need to impose
  that the equilibrium is unique.
\end{rmk}

     \section{Model with distributed delay: Proof of Theorems \ref{distri-conv-thm} and
     \ref{distri-conv-thm-noexp}}
\label{sec:distri-delay}

In this section we will consider the elapsed time model with
distributed delay given in \eqref{distri-delay-intro}
\begin{subequations}
  \label{distri-delay}
  \begin{align}
    \label{distri-delay-transport}
    &\partial_t n +\partial_a n + S(a,X(t))n=0,
    && t,a>0,
    \\
    \label{distri-delay-boundary}
    &n(t,a=0) = r(t) \coloneqq
      \int_0^\infty S(a,X(t))n(t,a)\da,
    && t>0,
    \\
    \label{distri-delay-X}
    &X(t)=\int_{-\infty}^t \alpha(t-s)r(s)\ds.
    && t>0.
  \end{align}
  \end{subequations}

Remind that in this case the equilibrium
distribution $(n^*, r^*)$ solves the system
\begin{equation*}
    \begin{cases}
           \partial_a n^* + S(a,X^*)n^*=0 & a>0,\vspace{0.1cm}\\
            r^*\coloneqq n^*(a=0)=\int_0^\infty S(a,X^*)n^*(a)\,da,&\vspace{0.1cm}\\
            X^*=r^*\int_0^\infty\alpha(s)\,ds.
    \end{cases}
\end{equation*}
where $n^*(a)=r^*e^{-\int_0^a S(a',r^*) \da'}$ and $r^*>0$ satisfies Equation \eqref{Int-r}.

For the proof of Theorems \ref{distri-conv-thm} and
\ref{distri-conv-thm-noexp} we first need the following comparison
lemma. For $f,g\in L^\infty(0,\infty)$ we use the notation
$f*g\coloneqq \int_0^t f(t-s)g(s) \d s$.

\begin{lem}
\label{comparison-delay-distri}
    Consider the functions $f,k\in L^\infty(0,\infty)$ with $k$ nonnegative. Let $\underline{u}\in L^\infty(0,\infty)$ such that
   \begin{equation}
         \underline{u}(t)\le(k*\underline{u})(t)+f(t)\qquad\forall t>0.
    \end{equation}
   and $\overline{u}\in L^\infty(0,\infty)$ such that 
   \begin{equation}
            \overline{u}(t)\ge (k*\overline{u})(t)+f(t)\qquad\forall t>0.
    \end{equation}
    Then it holds that $\underline{u}(t)\le\overline{u}(t)$ for all
    $t\ge0$.
\end{lem}

In other words $\underline{u}$ and $\overline{u}$ are respectively lower and upper solutions of the Volterra equation given by
\begin{equation}
 u(t)= (k*u)(t)+f(t)\quad \forall t>0,
\end{equation}
and the comparison principle holds.

\begin{proof}
  Observe that $h(t)\coloneqq\overline{u}(t)-\underline{u}(t)$
  satisfies
    $$h(t)\ge (k*h)(t).$$
    For $T>0$ we consider $A_1[T]\coloneqq\inf_{t\in[0,T]}h(t)$ and  we have
    $$\left(1-\int_0^T k(s) \ds \right)A_1[T]\ge 0.$$
    Therefore when we choose $T$ such that
    \begin{equation}
    \label{uniform-T}
      T\|k\|_\infty<1,  
    \end{equation}
    we conclude that $A_1[T]\ge0$, which means that $\overline{u}(t)\ge\underline{u}(t)$ for all $t\in [0,T]$. Similarly for $t\in[T,2T]$ we define
    $A_2[T]\coloneqq\inf_{t\in[T,2T]}$ and obtain
    $$
    \left(1-\int_T^{2T}k(s)\ds\right)A_2[T]\ge 0.
    $$
    Again, using the uniform bound for $T$ in
    \eqref{uniform-T} we have that
    $A_2[T]\ge0$, which implies that $\overline{u}(t)\ge\underline{u}(t)$ for all $t\in [T,2T]$. By iterating this argument, we deduce that $\overline{u}(t)\ge\underline{u}(t)$ for all $t\ge0$.
\end{proof}

Now we can prove Theorem \ref{distri-conv-thm}.

\begin{proof}[Proof of Theorem \ref{distri-conv-thm}]
  As in the proof of Theorem \ref{delay-conv-thm}, by Duhamel's
  formula and \Cref{doeblin-conv} there exist $C_0,\lambda>0$ such
  that the following inequality holds
    \begin{equation}
    \label{est-delay-n-n*2}
    \|n(t)-n^*\|_{TV}\le C_0e^{-\lambda t}\|n^0-n^*\|_{TV}+C_0\int_0^t e^{-\lambda(t-s)}\|h(s)\|_{TV}\ds,
    \end{equation}
    where $h$ is given by
    $$h(t,a)=(S(a,X^*)-S(a,X(t)))n(t,a)+\delta_0(a)\int_0^\infty
    (S(a,X(t))-S(a',X^*))n(t,a')\d a',$$
    thus, using Hypothesis \ref{hyp:S},  we have the estimate 
    \begin{equation}
      \label{est-h-distri}
      \|h(t,\cdot)\|_{TV}\le 2  \ell |X(t)-X^*|
      \qquad \text{for all $t>0$.}
    \end{equation}
    Also, we can estimate $|r(t)-r^*|$ as in the proof of Theorem \ref{delay-conv-thm}:
    \begin{multline}
      \label{eq:1}
      |r(t)-r^*|
      =
      \left| \int_0^\infty S(a, X(t)) n(t,a) \d a
        - \int_0^\infty S(a, X^*) n^*(a) \d a
      \right|
      \\
      \leq
      \int_0^\infty |S(a, X(t)) - S(a,X^*)| n(t,a) \d a
      +
      \int_0^\infty S(a, X^*) |n(t,a) - n^*(a)| \d a
      \\
      \leq
      \ell |X(t) - X^*|
      +
      \|S\|_\infty \| n(t,a) - n^*(a) \|_{TV}.
    \end{multline}
    Using now \eqref{est-delay-n-n*2} and \eqref{est-h-distri} we obtain
    \begin{multline}
      \label{eq:r-r*}
      |r(t)-r^*|
      \le
        \ell |X(t)-X^*|
      + C_0 \|S\|_\infty \|n^0-n^*\|_{TV}e^{-\lambda t}
      \\
      + 2C_0 \|S\|_\infty \ell \int_0^t e^{-\lambda(t-s)}|X(s)-X^*| \ds.
    \end{multline}
    To simplify the notation define the constants
    $C_1\coloneqq 2C_0\|S\|_\infty$ and
    $C_2\coloneqq C_0\|S\|_\infty \|n^0-n^*\|_{TV}$.

    We seek to estimate $|X(t)-X^*|$, so we define
    $u(t)\coloneqq |X(t)-X^*|$, and we obtain
    \begin{multline*}
      u(t) = \left|
        \int_0^\infty \alpha(s) r(t-s) \ds - r^* \int_0^\infty \alpha(s)
        \d s
      \right|
      \leq
      \int_0^\infty \alpha(s)|r(t-s)-r^*| \ds
      \\
      =
      \int_0^t \alpha(t-s)|r(s)-r^*| \ds
      + \int_t^\infty \alpha(s)|r^0(t-s)-r^*| \ds
      \\
      \leq
      \int_0^t \alpha(t-s)|r(s)-r^*| \ds
      + \|r^0-r^*\|_\infty \int_t^\infty \alpha(s) \ds.
    \end{multline*}
    Using \eqref{eq:r-r*} in the previous expression,
    \begin{multline*}
      u(t)
      \le \|r^0-r^*\|_\infty \int_t^\infty \alpha(s) \ds
      \\
      + \int_0^t \alpha(t-s) \left(
        \ell u(s) + C_2e^{-\lambda s} + C_1 \ell \int_0^{s} e^{-\lambda(s-s')}
        u(s') \d s'
      \right) \ds.
    \end{multline*}
    We define
    \begin{equation*}
      g(t)\coloneqq \|r^0-r^*\|_\infty
      \int_t^\infty \alpha(s) \ds
      + C_2\int_0^t\alpha(t-s)e^{-\lambda s}\ds,
    \end{equation*}
    so we write the inequality for $u(t)$ as
    \begin{equation*}
      u(t) \le g(t)+ \ell (\alpha*u) + C_1 \ell (\alpha*e^{-\lambda t}*u).
    \end{equation*}

    Like in the case of a single discrete delay, we aim to apply the
    comparison lemma. We look for constants $A,\mu>0$ such that
    $u(t)\le Ae^{-\mu t}$ for all $t\ge0$. For this, we would like the
    function $v(t)\coloneqq Ae^{-\mu t}$ to satisfy
    \begin{equation*}
      v(t)\ge g(t)+\ell(\alpha*v)
      +C_1\ell(\alpha*e^{-\lambda t}*v)
      \qquad \text{for all $t\ge0$,}
    \end{equation*}
    or equivalently in terms of $A$ and $\mu$
    \begin{equation}
    \label{supersol-A-lambda-distri}
    A \geq
    g(t) e^{\mu t}
    + \ell A\int_0^t e^{\mu s}\alpha(s) \ds
    + \ell \frac{AC_1}{\lambda-\mu}\int_0^t e^{\mu
      s}\alpha(s)(1-e^{-(\lambda-\mu)(t-s)})
    \ds
    \qquad \text{for $t\ge0$.}
    \end{equation}
    For $\mu<\min\{\beta,\lambda\}$, we estimate each term in the
    right-hand side. For the first one,
    \begin{gather*}
      g(t)e^{\mu t}\le C_3 \frac{C_\alpha}{\beta} e^{-(\beta-\mu)t}
      +C_2 C_\alpha\frac{e^{-(\beta-\mu) t}-e^{-(\lambda-\mu)
          t}}{\lambda-\beta}
      \le
      C_\alpha\left(\frac{C_3}{\beta}+\frac{C_2}{|\lambda-\beta|}\right),
    \end{gather*}
    where we call $C_3 := \|r^0-r^*\|_\infty$. For the remaining two
    terms we have
    \begin{gather*}
    \int_0^t e^{\mu s}\alpha(s)\ds
    \le
    C_\alpha\frac{1-e^{-(\beta-\mu)t}}{\beta-\mu}\le
    \frac{C_\alpha}{\beta-\mu},
     \\
     \int_0^t e^{\mu s}\alpha(s)(1-e^{-(\lambda-\mu)(t-s)})\ds
     \le
     \int_0^t e^{\mu s}\alpha(s)\ds
     \le \frac{C_\alpha}{\beta-\mu}.
   \end{gather*}
   Hence in order to satisfy \eqref{supersol-A-lambda-distri} it is
   enough to satisfy
   \begin{equation*}
     A \geq
     C_\alpha\left(\frac{C_3}{\beta}+\frac{C_2}{|\lambda-\beta|}\right)
     + \frac{\ell A C_\alpha}{\beta-\mu} \left(1 + \frac{C_1}{\lambda-\mu}\right),
   \end{equation*}
   that is,
   \begin{equation*}
     A \left(1 -
       \frac{\ell C_\alpha}{\beta-\mu} \left(1 +
         \frac{C_1}{\lambda-\mu}
       \right)
     \right)
     \geq
     C_\alpha\left(\frac{C_3}{\beta}+\frac{C_2}{|\lambda-\beta|}\right).
   \end{equation*}
   Therefore if the following inequalities hold
   \begin{gather*}
     \ell <
     \frac{\beta-\mu}{C_\alpha}\frac{\lambda-\mu}{\lambda-\mu+C_1},
     \\
     A >
     C_\alpha\left(\frac{C_3}{\beta}+\frac{C_2}{|\lambda-\beta|}\right)
     \left(\frac{(\beta-\mu)(\lambda-\mu)}{(\beta-\mu)(\lambda-\mu)-
         \ell C_\alpha(\lambda-\mu+C_1)}
     \right)
   \end{gather*}
   we get that $A$ and $\mu$ satisfy \eqref{supersol-A-lambda-distri}
   and thus, due to our comparison result in Lemma
   \ref{comparison-delay-distri}
   \begin{equation*}
     |X(t)-X^*|\le A e^{-\mu t}
     \qquad \text{for $t \ge0$}.
   \end{equation*}
   Notice that the dependence on $\|r^0 - r^*\|_\infty$ and
   $\|n^0-n^*\|_{\mathrm{TV}}$ are included in $C_3$ and $C_2$,
   respectively. The exponential decay of $\|n(t)-n^*\|_{TV}$
   readily follows from \eqref{est-delay-n-n*2} and
   \eqref{est-h-distri}, and then exponential decay of $|r(t) - r^*|$
   follows from \eqref{eq:1}.  
\end{proof}

To prove Theorem \ref{distri-conv-thm-noexp} regarding the case in
which $\alpha$ decays algebraically we will need the following lemma
on decay of convolutions:

\begin{lem}
  \label{decay-conv}
  Let $f,g\in L^\infty(\R^+)$ and $a>0,b>1$ such that $f=O(t^{-a})$ and
  $g=O(t^{-b})$ when $t\to\infty$. Then for their convolution we have
  \begin{equation*}
    h(t)\coloneqq\int_0^t f(t-s)g(s)\ds=O(t^{-\min\{a,b-1\}})
    \qquad \text{as $t\to\infty$.}
  \end{equation*}
\end{lem}

\begin{proof}
  Observe that $g\in L^1(\R^+)$ since $b>1$. Thus there exists two
  constants $C_1,C_2>0$ such that for $t$ large enough we have the
  following estimate
    \begin{equation*}
        \begin{split}
            |h(t)|&\le \int_0^t |f(t-s)g(s)|\ds,\\
            &\le \int_0^\frac{t}{2} |f(t-s)g(s)|\ds+\int_\frac{t}{2}^t |f(t-s)g(s)|\ds\\
            &\le C_1\int_0^\frac{t}{2} (t-s)^{-a}|g(s)|\ds+C_2\int_\frac{t}{2}^t |f(t-s)|s^{-b}\ds\\
            &\le 2^a t^{-a}C_1\|g\|_1+C_2\|f\|_\infty\frac{2^{b-1}-1}{b-1}t^{-(b-1)},
        \end{split}
    \end{equation*}
    where the last inequality proves the desired result.
\end{proof}

With this lemma we  prove  Theorem \ref{distri-conv-thm-noexp}.

\begin{proof}[Proof of Theorem \ref{distri-conv-thm-noexp}]
  We can carry out the same initial steps as in the exponential
  case. With the same notation, the function $u(t)=|X(t)-X^*|$
  satisfies
  $$u(t) \le g(t)+ \ell (\alpha*u) + C_1 \ell (\alpha*e^{-\lambda t}*u),$$
  with
  $g(t)=C_3 \int_t^\infty\alpha(s) \ds+C_2\int_0^t\alpha(t-s)e^{-\lambda
    s}\ds$. We recall that the constants $C_1$, $C_2$ and $C_3$ were
  defined by
  \begin{equation*}
    C_1\coloneqq 2C_0\|S\|_\infty,
    \qquad
    C_2\coloneqq C_0\|S\|_\infty \|n^0-n^*\|_{TV},
    \qquad
    C_3 := \|r^0-r^*\|_\infty.
  \end{equation*}
  Like the previous result, we aim to apply the comparison lemma. We
  look for constants $A,\mu>0$ such that the function
  $v(t)\coloneqq \frac{A}{1+t^\mu}$ satisfies the inequality
  \begin{equation*}
    v(t)\ge g(t)+ \ell (\alpha*v) + C_1 \ell (\alpha*e^{-\lambda t}*v)
    \qquad \text{for all $t\ge0$},
  \end{equation*}
  or equivalently in terms of $A$ and $\mu$,
  \begin{multline}
    \label{supersol-A-lambda-distri2}
    A \geq
    g(t)(1+t^\mu)
    + \ell A\int_0^t \frac{1+t^\mu}{1+(t-s)^\mu}\alpha(s) \ds
    \\
    + \ell A C_1 \int_0^t\int_0^s
    \alpha(t-s)e^{-\lambda(s-s')}\frac{1+t^\mu}{1+(s')^{\mu}}\ds'\ds
  \end{multline}
  for all $t \geq 0$. We now estimate each term in the right-hand
  side. First observe that for the first term of $g(t)$ we have that
  \begin{equation*}
    \int_t^\infty \alpha(s)\ds \leq \frac{C_{\alpha,\beta}}{1 + t^{\beta-1}}
  \end{equation*}
  for some constant $C_{\alpha,\beta} > 0$ depending on $C_\alpha$ and
  $\beta$. Thus, by choosing $\mu=\beta-1$ and applying Lemma
  \ref{decay-conv}, there exists a constant $C_4 > 0$ depending on
  $C_\alpha$ and $\beta$ such that
  \begin{equation*}
    g(t)(1+t^\mu)
    \leq
    C_3 C_4 \frac{1+t^\mu}{1+t^{\beta-1}}
    + C_2 C_4 \frac{1+t^\mu}{1+t^{\beta-1}}
    \leq
    C_4 (C_2 + C_3)
  \end{equation*}
  and similarly (possibly taking a larger constant $C_4$) we get
  \begin{align*}
    &\int_0^t \frac{1+t^\mu}{1+(t-s)^\mu}\alpha(s)\ds
    \leq
    C_4,
    \\
    &\int_0^t\int_0^s
      \alpha(t-s)e^{-\lambda(s-s')}\frac{1+t^\mu}{1+{s'}^{\mu}}\ds'\ds
      \leq
      C_4.
  \end{align*}
  Therefore in order to satisfy \eqref{supersol-A-lambda-distri} it is
  enough to satisfy
  \begin{equation*}
    A \geq C_4 (C_2 + C_3) + \ell A C_4 (1 + C_1),
  \end{equation*}
  or equivalently
  \begin{equation*}
    A (1 - \ell C_4 (1 + C_1) ) \geq C_4 (C_2 + C_3).
  \end{equation*}
  Hence, if the following inequalities hold 
  \begin{equation*}
    \ell < \frac{1}{C_4(1+C_1)},
    \qquad
    A > \frac{C_4 (C_2 + C_3)}{1-\ell C_4 (1+C_1)}.
  \end{equation*}
  we get that $A$ and $\mu$ satisfy \eqref{supersol-A-lambda-distri2}
  and thus
  \begin{equation*}
    |X(t)-X^*|\le \frac{A}{1+t^{\beta-1}}
    \qquad \text{for all $t\ge0$}.
  \end{equation*}
Notice again that the dependence on the initial condition is
  implicit in $C_2$ and $C_3$. The convergence of $|r(t) - r^*|$ and
  $\|n(t)-n^*\|_{TV}$ readily follows from estimates
  \eqref{est-delay-n-n*2}, \eqref{est-h-distri} and \eqref{eq:1} as in
  the exponential case, by using Lemma \ref{decay-conv} to estimate the
  integral in \eqref{est-delay-n-n*2}.
\end{proof}

We end the paper with the following
two remarks.
\begin{rmk}
  The convergence results of Theorems \ref{distri-conv-thm} and
  \ref{distri-conv-thm-noexp} with $\alpha$ bounded by an exponential
  function or with algebraic tail, respectively, can be extended for a
  general $\alpha$ as long as we are able to find a suitable upper
  solution, which might depend on several parameters and an
  optimization may be performed.
\end{rmk}

\begin{rmk}[Relaxed hypotheses]
\label{rmkDistri}
Analogously to Remark \ref{rmkDiscret}, our proofs for the model with
distributed delay can be carried out under the relaxed Hypothesis
\ref{hyp: S-weak-strong} on $S$, instead of Hypothesis \ref{hyp:S}.
In this case we obtain convergence to the equilibrium in both regimes,
provided the initial data is close to the equilibrium in terms of $r$.
\end{rmk}

\section*{Acknowledgments}

The authors acknowledge support from projects of the Spanish
\emph{Ministerio de Ciencia e Innovación} and the European Regional
Development Fund (ERDF/FEDER) through grants PID2023-151625NB-100,
RED2022-134784-T, and CEX2020-001105-M, all of them funded by
MCIN/AEI/10.13039/501100011033. NT was supported by the grant Juan de
la Cierva FJC2021-046894-I funded by MCIN/AEI with the European Union
NextGenerationEU/PRTR, and also supported by the project EUR SPECTRUM
with the initiative IDEX of Université de Côte d'Azur.

\bibliography{Delay_model.bib}
\bibliographystyle{plain}

\end{document}